\documentclass{ws-rmp}
\graphicspath{{figures/}}
\usepackage[T1]{fontenc}
\usepackage{dsfont}
\usepackage{booktabs}
\usepackage{amsmath,amssymb}
\usepackage{graphicx}
\usepackage{placeins}
\makeatletter
\let\input@path\Ginput@path
\makeatother
\usepackage{float}
\usepackage{mathtools}
\usepackage{bm}
\usepackage{hyperref}
\usepackage{newunicodechar}
\newunicodechar{–}{--}
\newunicodechar{—}{---}

\newcommand{\rmpincludegraphics}[2][]{%
  \IfFileExists{#2}{\includegraphics[#1]{#2}}{%
    \IfFileExists{#2.pdf}{\includegraphics[#1]{#2.pdf}}{%
      \IfFileExists{#2.png}{\includegraphics[#1]{#2.png}}{%
        \fbox{\parbox{0.9\linewidth}{\centering Missing figure: \detokenize{#2}}}%
      }%
    }%
  }%
}
\newcommand{\rmpbenchsheetgraphics}[1]{%
  \rmpincludegraphics[width=\linewidth]{#1}%
}
\newcommand{\rmpvroverlaygraphics}[1]{%
  \rmpincludegraphics[width=\linewidth]{#1}%
}

\newcommand{\benchfigCoulomb}{rmp_bench_laplace_p6gbm_coulomb}
\newcommand{\benchfigHO}{rmp_bench_laplace_p6gbm_ho}
\newcommand{\benchfigHulthen}{rmp_bench_laplace_p6gbm_hulthen}
\newcommand{\benchfigKratzer}{rmp_bench_laplace_p6gbm_kratzer}
\newcommand{\benchfigHyperbolic}{rmp_bench_laplace_p6gbm_hyperbolic}

\newcommand{\benchvrfigCoulomb}{bench_vr_lsq_p6gbm_coulomb}
\newcommand{\benchvrfigHO}{bench_vr_lsq_p6gbm_ho}
\newcommand{\benchvrfigHulthen}{bench_vr_lsq_p6gbm_hulthen}
\newcommand{\benchvrfigKratzer}{bench_vr_lsq_p6gbm_kratzer}
\newcommand{\benchvrfigHyperbolic}{bench_vr_lsq_p6gbm_hyperbolic}

\makeatletter
\@ifundefined{catchline}{}{\renewcommand{\catchline}[5]{}}
\makeatother
\makeatletter
\def\trimmarks{}
\makeatother

\makeatletter

\providecommand*{\theHremark}{}
\renewcommand*{\theHremark}{\thesection.\arabic{remark}}
\@ifundefined{c@Remark}{}{}
\makeatother
\makeatletter
\renewcommand{\theHequation}{\thesection.\arabic{equation}}
\makeatother

\hypersetup{hidelinks,hypertexnames=false}
\makeatletter
\@ifundefined{remark}{%
\newenvironment{remark}[1][]{\refstepcounter{equation}\par\noindent\textbf{Remark\if\relax\detokenize{#1}\relax\else\ (#1)\fi.}\ }{\par}
}{}
\makeatother

\makeatletter
\@ifundefined{theorem}{}{%
 \@ifundefined{Theorem}{\newenvironment{Theorem}{\begin{theorem}}{\end{theorem}}}{}%
}
\@ifundefined{lemma}{}{%
 \@ifundefined{Lemma}{\newenvironment{Lemma}{\begin{lemma}}{\end{lemma}}}{}%
}
\@ifundefined{corollary}{}{%
 \@ifundefined{Corollary}{\newenvironment{Corollary}{\begin{corollary}}{\end{corollary}}}{}%
}
\@ifundefined{proposition}{}{%
 \@ifundefined{Proposition}{\newenvironment{Proposition}{\begin{proposition}}{\end{proposition}}}{}%
}
\@ifundefined{definition}{}{%
 \@ifundefined{Definition}{\newenvironment{Definition}{\begin{definition}}{\end{definition}}}{}%
}
\@ifundefined{remark}{}{%
 \@ifundefined{Remark}{\newenvironment{Remark}{\begin{remark}}{\end{remark}}}{}%
}
\@ifundefined{example}{}{%
 \@ifundefined{Example}{\newenvironment{Example}{\begin{example}}{\end{example}}}{}%
}
\makeatother

\providecommand{\proofpart}[1]{\par\smallskip\noindent\textit{#1.}\,}
\makeatletter
\@ifundefined{theorem}{}{%
 \@ifundefined{Theorem}{}{}%
}
\@ifundefined{lemma}{}{%
 \@ifundefined{Lemma}{}{}%
}
\@ifundefined{corollary}{}{%
 \@ifundefined{Corollary}{}{}%
}
\@ifundefined{proposition}{}{%
 \@ifundefined{Proposition}{}{}%
}
\@ifundefined{definition}{}{%
 \@ifundefined{Definition}{}{}%
}
\@ifundefined{remark}{}{%
 \@ifundefined{Remark}{}{}%
}
\@ifundefined{example}{}{%
 \@ifundefined{Example}{}{}%
}
\makeatother
\begin{document}
\markboth{Plott, \c{C}etinkaya \& Mukherjee}{Inverse Potential Reconstruction via GBM}

%%%%%%%%%%%%%%%%%%%%% Publisher's Area please ignore %%%%%%%%%%%%%%%
% \catchline{}{}{}{}{}  % (commented to hide template timestamp/jobname)
%%%%%%%%%%%%%%%%%%%%%%%%%%%%%%%%%%%%%%%%%%%%%%%%%%%%%%%%%%%%%%%%%%%%

\title{Inverse Quantum Potential Reconstruction via Generalized Bertlmann--Martin Inequalities}

\author{M.\ Gage Plott}
\address{Department of Mathematics, University of Tennessee at Chattanooga, Chattanooga, TN 37403, USA\\
\email{mgplott13@gmail.com}}

\author{F.\ Ay\c{c}a \c{C}etinkaya}
\address{Department of Mathematics, University of Tennessee at Chattanooga, Chattanooga, TN 37403, USA\\
\email{fatmaayca-cetinkaya@utc.edu}}

\author{Rick Mukherjee}
\address{Department of Physics, University of Tennessee at Chattanooga, Chattanooga, TN 37403, USA\\
\email{rick-mukherjee@utc.edu}}

\maketitle

% Draft-share copy: omit template received/revised placeholders.

\begin{abstract}
Reconstructing a radial (1D) quantum potential, $V(r)$, from a few bound-state energies is a long-standing inverse problem because limited spectral data must constrain an entire potential. We present a Laplace-moment reconstruction pipeline that links the Bertlmann--Martin gap bound to generalized Bertlmann--Martin (GBM) even-moment ladders, continues the Laplace transform with Pad\'e approximants, and inverts the transform to recover $\rho(r)$ and $V(r)$. Odd moments are supplied by a physically consistent interpolation scheme. Benchmark settings and diagnostics for Coulomb, harmonic oscillator, Hulth\'en, Kratzer, and hyperbolic-well cases are stated so each approximation stage can be assessed under a common empirical basis. The conclusions are therefore limited to the reported benchmark settings rather than offered as universal method claims.
\end{abstract}


%% --- Body ---

\section{Introduction}

\label{sec:intro}
An inverse problem aims to recover the underlying cause from observed (measured) effects. These problems are generally ill-posed, in the sense that small changes in the data may result in significant changes in the corresponding solution \cite{DoerflerHochbruckKoehlerRiederSchnaubeltWieners2023}. Despite this inherent difficulty, a number of fundamental results demonstrate that meaningful reconstruction is possible under suitable conditions.

One of the earliest results in this direction is due to Ambarzumian \cite{Ambarzumian}, who showed that if the eigenvalues of the boundary value problem
\begin{equation}
  -y^{\prime \prime}+V(x)y=\lambda y, \ y^\prime (0)=y^\prime (\pi)=0
\label{eq:SL-BVP-intro}
\end{equation}
are $\lambda_k = k^2$, $k \geq 0$, then $V=0$. However, this result is exceptional, and in general the specification of spectral data does not uniquely determine the operator. Subsequent results by Borg \cite{Borg} and Levinson \cite{Levinson} established uniqueness under additional spectral information, while Tikhonov \cite{Tikhonov} obtained uniqueness on the half-line via the Weyl function.

Questions of existence and uniqueness for inverse spectral problems have been studied extensively over the past decades \cite{Levitan1987InverseSturmLiouville,FreilingYurko2001InverseSturmLiouville,MamedovCetinkaya2013SpectralParameter,Bondarenko2015MatrixHalfLine,Chen2016InteriorTransmission,ButerinSat2017HalfInverseIntegro,DelgadoKhmelnytskayaKravchenko2019,ZhangBondarenkoYang2021Discontinuous,VladicicBoskovicVojvodic2022Delay,WeiHuXiang2024Reconstruction,GawishMansour2025QSturmLiouville,FeizmohammadiKian2026PartialBoundary}. Alongside these developments, a wide range of numerical reconstruction methods has been developed, including constructive approaches based on the Gel'fand--Levitan--Marchenko (GLM) integral equations \cite{GelfandLevitan,Marchenko,HronRazavy1983GLMApplications,Habashy1991GeneralizedGLM,KravchenkoTorba2021DirectMethod,KravchenkoVicenteBenitez2022TransmutationII,BaevMozgovykh2025MethodSolving}. Although GLM-based methods rest on a rigorous analytical foundation, their practical implementation is often limited by the sensitivity of the integral kernels to incomplete spectral data. Another widely used approach is inverse scattering \cite{Chadan,Newton,KravchenkoShishkinaTorba2020InverseScattering,GradusovYakovlev2023InverseSquareScattering,Gibson2026ExplicitInverseScattering}, in which the potential is reconstructed from scattering data, such as reflection coefficients or phase shifts, through integral equation techniques; however, in many applications, such data may be incomplete or inaccessible. Alternatively, discrete approaches, such as finite difference and spectral methods, reformulate the differential operator as a matrix-based inverse eigenvalue problem \cite{FabianoKnobelLowe1995FiniteDifference,MarlettaWeikard2005WeakStability,FreilingMazurYurko2007Singular,Andrew2011FiniteDifference,GaoHuangCheng2015FiniteDifference,Kravchenko2020DirectInverseBook,Bondarenko2022FrozenArgument,CetinkayaKhmelnytskayaKravchenko2024NSBF,Kravchenko2025WeylHalfLine}; however, accurate reconstruction often demands fine discretizations, which increase computational cost and potentially affect numerical stability.

In contrast, moment-based reconstruction methods \cite{Bertlmann,Mezhoud2003GBM,Yekken,Mezhoud} provide an alternative framework in which the potential is recovered through the moments of the ground-state density derived from spectral information. By avoiding the need for scattering data, large-scale discretization, or the solution of integral equations, these methods offer a comparatively direct and flexible approach to inverse problems. Moreover, accurate reconstruction can often be achieved from a finite set of spectral quantities, making moment-based techniques particularly attractive in settings where available data are limited.

Section~\ref{sec:methodology} develops the pipeline, and Section~\ref{sec:benchmark} benchmarks it on Coulomb, harmonic oscillator, Kratzer, Hulth\'en, and hyperbolic molecular well potentials.
\section{Problem Setting and Preliminaries}
\label{sec:methodology}

In this section, we introduce the mathematical framework underlying the inverse problem and establish the notation used throughout the paper. We begin with the Schr\"odinger equation under spherical symmetry, then define the ground-state density and its associated moments, and finally introduce the transform representation that forms the basis of the reconstruction method.

We consider the non-relativistic Schr\"odinger equation in 3D with spherical symmetry for a local, central potential \(V(r)\)
\begin{equation}
\label{eq:schrodinger_3d_TI}
\left[
-\frac{1}{2}\nabla^{2} + V(\mathbf{r})
\right]
\Psi(\mathbf{r})
= E\,\Psi(\mathbf{r}).
\end{equation}

Denote the ground-state radial wave by \(R_{0,0}(r)\), the reduced radial wavefunction by
\[
\chi_{0,0}(r):=r\,R_{0,0}(r),
\]
and the radial density by
\begin{equation}
 \rho_{0,0}(r) := |R_{0,0}(r)|^2, \qquad r\ge 0,
 \label{eq:rho_def}
\end{equation}
From these definitions, it can be seen that
\[
\int_0^\infty r^2 \rho_{0,0}(r)\,dr = 1.
\]
Define the ground-state moments as
\begin{equation}
 \langle r^n \rangle_{0,0} := \int_0^\infty r^{n+2}\,\rho_{0,0}(r)\,dr,
 \qquad n=0,1,2,\dots,
 \label{eq:moment_def}
\end{equation}
The Laplace transform of \(r^2\rho_{0,0}(r)\) is given by
\begin{equation}
 L(q) := \int_0^\infty e^{-q r}\, r^{2}\rho_{0,0}(r)\,dr
 = \sum_{n\ge 0} \frac{(-1)^n}{n!}\, \langle r^n\rangle_{0,0}\, q^n,
 \qquad q\ge 0.
 \label{eq:L_series}
\end{equation}
The inverse problem is centered on recovering the reduced density \(r^2\rho_{0,0}(r)=\chi_{0,0}^2(r)\) from its Laplace transform. To compute \(L(q)\) without knowing the wave function a priori, one needs a way to find the moments \(\langle r^n\rangle_{0,0}\) from accessible spectral data.

These notation and unit conventions fix the Hartree atomic-units basis used
below, so the active manuscript equations are written directly in Hartree form
with kinetic term \(-\tfrac{1}{2}\nabla^2\). Source formulas written in other
nondimensionalizations are rescaled to this convention before use. To calculate
these moments from accessible spectral data, we rely on the formalism
established in Appendix~A. Following Theorem~1, we perform a radial reduction
\(u(r)=rR(r)\) to map the 3D Schr\"odinger operator onto a radial (1D) form. This
reduction defines the ground-state moments as radial (1D) spatial integrals and,
through Theorem~1, ensures a rigorous ordering of these moments relative to the
system's energy levels while preserving the Hartree kinetic prefactor used in
the active manuscript equations.

State labels use radial--orbital indices \((n_r,\ell)\) throughout; the lowest
channels used here are \((0,0),(0,1),(1,0),(0,2)\).
When both symbols appear, \(E_{0,\ell}\) denotes the yrast level in the
\(\ell\)-channel (\(n_r=0\)), while \(E_{\ell,0}\) denotes the \(\ell\)-th
radial excitation in the \(\ell=0\) channel used in \(f(\ell)\).
Moments of the ground-state density are
\begin{equation}
\langle r^k\rangle_{0,0}=\int_{0}^{\infty}\rho_{0,0}(r)\,r^{k+2}\,dr
=\int_{0}^{\infty}|R_{0,0}(r)|^2\,r^{k+2}\,dr.
\label{eq:moments-def}
\end{equation}
Appendix A gives derivations of the equations that follow, especially the
change of variables required to go from 3D to the radial (1D) reduction.

With this framework established, we calculate the moment sequence using the
generalized Bertlmann--Martin (GBM) inequalities. The primary challenge in
constructing the series in \eqref{eq:L_series} is that the ground-state wave
function, and thus the moments, are unknown. However, GBM inequalities allow us
to bypass the wave equation by establishing a recursive relationship between the
even moments \(\langle r^{2\ell}\rangle_{0,0}\) and the bound-state energy gaps.
The construction of the even-moment sequence relies on the sum rule for the
operator \(Q_{\ell,0}(\mathbf r)=r^\ell Y_{\ell 0}(\theta,\phi)\), where
\(Y_{\ell 0}\) denotes the usual spherical harmonic:
\begin{equation}
(E_{0,\ell}-E_{0,0})\sum_{n_r}\big|\langle 0,0|Q_{\ell,0}|n_r,\ell\rangle\big|^2
\;\le\;\sum_{n_r}(E_{n_r,\ell}-E_{0,0})\big|\langle 0,0|Q_{\ell,0}|n_r,\ell\rangle\big|^2,
\end{equation}
Here the squared matrix elements
\(\big|\langle 0,0|Q_{\ell,0}|n_r,\ell\rangle\big|^2\) are the corresponding
multipole transition strengths, and the factors \(E_{n_r,\ell}-E_{0,0}\) are
energy gaps from the ground state.
By evaluating the sums in this inequality, the Bertlmann--Martin sum-rule evaluation in the Hartree convention reduces to the recursive moment-ladder bound (for \(\ell\ge 1\)):
\begin{equation}
\boxed{\;
\langle r^{2\ell}\rangle_{0,0}\;\le\;
\frac{1}{2}\,\frac{\ell(2\ell+1)}{E_{0,\ell}-E_{0,0}}\;\langle r^{2\ell-2}\rangle_{0,0}\; }.
\label{eq:BM-ineq}
\end{equation}
\begin{equation}
E_{0,\ell}-E_{0,0}\;\le\;\frac{1}{2}\,\frac{\ell(2\ell+1)\,\langle r^{2\ell-2}\rangle_{0,0}}{\langle r^{2\ell}\rangle_{0,0}}.
\label{eq:BM-gap-upper}
\end{equation}
This energy-gap bound serves as a consistency constraint: the measured energy
levels used in the ladder must remain compatible with the recovered density
moments.

While the GBM relations established in \eqref{eq:BM-ineq} and
\eqref{eq:BM-gap-upper} provide rigorous upper bounds on the even-moment
sequence, the inverse reconstruction process requires precise numerical values
rather than inequality ranges. To bridge this gap, a saturating correction
factor
\begin{equation}
f(\ell)=1-\frac{\ell}{2(\ell+1)}
\left[\frac{E_{\ell,0}+E_{0,0}-2E_{0,\ell}}{E_{\ell,0}-E_{0,0}}\right]^2,
\label{eq:f-ell}
\end{equation}
is introduced. Multiplying the right-hand side of \eqref{eq:BM-ineq} by this factor yields a working moment estimate that is exact for the harmonic oscillator and Coulomb potentials and serves as a reliable ansatz for other potentials. This process successfully populates the even parity of the Laplace series expansion in \eqref{eq:L_series}. However, due to the structural limitations of the underlying sum rules, the odd-moment subsequence remains undetermined by the GBM ladder and must be supplied via auxiliary interpolation to complete the transform's representation.

By mapping measured energy gaps to the recursive bound in \eqref{eq:BM-ineq}
and applying the correction factor \(f(\ell)\), we obtain the coefficients
necessary to construct the even parity of the Laplace series in
\eqref{eq:L_series}. The odd parity for \eqref{eq:L_series} must be determined
through an auxiliary interpolation or completion scheme \cite{Mezhoud}.%

Because the truncated series \eqref{eq:L_series} is a local expansion about \(q=0\), it can be unsuitable for the \(q\)-values needed for stable inversion. To stabilize the transform, the truncated series is replaced with a Pad\'e rational approximant \(P(N,D)\). In practice, one constructs several low-order denominator-dominant pairs \((N,D)\) whose degree excess is compatible with the large-\(q\) decay of \(L(q)\), (i) discards those with obstructive singularities, and (ii) averages over the acceptable fractions to reduce dispersion at large \(q\).

Once a stable representation of \(L(q)\) is established through a Pad\'e \(P(N,D)\), the ground-state density \(\chi_{0,0}^2(r)\) is recovered through the Bromwich integral
\begin{equation}
\chi_{0,0}^2(r)=r^2\rho_{0,0}(r)=\frac{1}{2\pi i}\int_{c-i\infty}^{c+i\infty}
e^{qr}\,L(q)\,dq,\qquad c>\Re(\text{all singularities}) .
\label{eq:Bromwich}
\end{equation}
For simple rational forms, this inversion can be executed analytically using residue theory. In more complex cases, numerical Bromwich-inversion schemes such as de Hoog--Knight--Stokes Fourier-series acceleration and Abate--Whitt variants \cite{deHoogKnightStokes1982,AbateWhitt1995} can be used to invert the Laplace transform \(L(q)\) and recover the ground-state density. A comprehensive numerical assessment of these inversion schemes is presented in Section~\ref{sec:benchmark}, while additional benchmark results are detailed in Appendix~C.

Upon obtaining the ground-state density \(\chi_{0,0}^2(r)\) from \eqref{eq:Bromwich}, the reduced wavefunction \(\chi_{0,0}(r)\) is determined by taking the positive square root. The radial ground state is then given by
\begin{equation}
R_{0,0}(r)=\frac{\chi_{0,0}(r)}{r},\qquad r>0.
\end{equation}
At very small \(r\), the numerical inversion may amplify numerical noise and endpoint artifacts; in such cases, local interpolation techniques can be applied to ensure a smooth transition toward the origin.%

Finally, for the \(\ell=0\) ground state, the underlying potential \(V(r)\) is recovered by inverting the Schr\"odinger equation in Hartree a.u.:
\begingroup
\setlength{\abovedisplayskip}{8pt}
\setlength{\belowdisplayskip}{8pt}
\setlength{\abovedisplayshortskip}{6pt}
\setlength{\belowdisplayshortskip}{6pt}
\begin{equation}
V(r)=E_{0,0}+\frac{1}{2}\,\frac{\chi_{0,0}''(r)}{\chi_{0,0}(r)},\qquad r>0.
\label{eq:V-from-chi}
\end{equation}
\endgroup
To maintain numerical stability as \(r\to0\), the moments \(\langle r^{-1}\rangle\) and \(\langle r^{-2}\rangle\) derived from the Laplace transform \(L(q)\) are utilized to constrain the behavior of \(\rho(0)\) via a Stieltjes-based approach \cite{Stieltjes1894FractionsContinues}. This ensures that the reconstructed potential remains physically consistent near the origin.

\section{Benchmark Results} \label{sec:benchmark}
\subsection*{Radial (1D) Simulations} \label{subsec:1D-sims}
We evaluate the Laplace--GBM framework across five radial benchmarks to determine how accurately ground-state structure can be recovered from a limited spectral subset. To provide a consistent basis for evaluation, each benchmark sheet (Figs.~\ref{fig:coulomb}--\ref{fig:hyperbolic-dw}) follows a standardized four-panel diagnostic layout. The top row presents the physical reconstruction, comparing the recovered potential \(V(r)\) and the radial density \(r^2\rho_{0,0}(r)\) against exact or reference benchmark surfaces. The bottom row monitors the mathematical stability of the transform-based ladder by tracking the Pad\'e-continued Laplace transform \(L(q)\) and the resulting reduced radial wavefunction \(\chi_{0,0}(r)\). Throughout these tests, we compare our results against the variational least-squares baseline of R{\"o}hrl \cite{Roehrl2006RecoveringBC,Roehrl2011LeastSquares}. For cross-potential comparability, the R{\"o}hrl-style LSQ settings are calibrated on the Coulomb potential and then transferred unchanged to the remaining four potentials.

% Tighten spacing in the figure-heavy benchmark block without changing global layout.
\begingroup
\setlength{\intextsep}{2pt}
\setlength{\textfloatsep}{2pt}
\setlength{\floatsep}{2pt}
\setlength{\abovecaptionskip}{2pt}
\setlength{\belowcaptionskip}{0pt}
\begin{figure}[H]
 \centering
 \rmpbenchsheetgraphics{\benchfigCoulomb}
 \caption{Coulomb benchmark sheet.}
 \label{fig:coulomb}
\end{figure}

The Coulomb and harmonic oscillator (HO) cases serve as the analytical anchors for this study, as the GBM saturation factor is exact for these specific potentials. In the Coulomb benchmark, the pipeline achieves a high-resolution reconstruction with a relative \(L^2\) error of \(1.91\times10^{-3}\) while consuming only seven recurrence-consumed spectral inputs. This performance highlights the method's efficiency; it provides a reliable fit with approximately \(94.2\%\) fewer inputs than the 120-constraint R{\"o}hrl LSQ baseline. Similarly, the harmonic oscillator demonstrates the pipeline's effectiveness when the transform \(L(q)\) is well behaved, accurately reproducing the benchmark surfaces with a relative error of \(7.10\times10^{-3}\) under the non-Coulomb input protocol.

\begin{figure}[H]
 \centering
 \rmpbenchsheetgraphics{\benchfigHO}
 \caption{Harmonic oscillator benchmark sheet.}
 \label{fig:ho}
\end{figure}

Moving beyond these saturated cases, the Hulth\'en and Kratzer benchmarks test the method's ability to handle potentials with different asymptotic decays while maintaining Coulomb-like near-origin structure. These benchmarks utilize a Pad\'e-complete odd-moment closure to supply the missing spectral information required by the Laplace series. The Kratzer reconstruction is particularly stable, yielding an \(L^2\) error of \(4.94\times10^{-3}\) from 22 inputs. The Hulth\'en case further confirms this stability, capturing the qualitative structure of the density and potential with an error of \(8.95\times10^{-3}\).

\begin{figure}[H]
 \centering
 \rmpbenchsheetgraphics{\benchfigHulthen}
 \caption{Hulth\'en benchmark sheet.}
 \label{fig:hulthen}
\end{figure}

\begin{figure}[H]
 \centering
 \rmpbenchsheetgraphics{\benchfigKratzer}
 \caption{Kratzer benchmark sheet.}
 \label{fig:kratzer}
\end{figure}

\enlargethispage{\baselineskip}
The hyperbolic molecular well serves as the primary stress test for the pipeline, introducing anharmonic structures that are more sensitive to the truncated series representation. Unlike the previous cases, this benchmark utilizes maximum-entropy closure for odd-moment completion. While the qualitative structure of the potential is recovered, the relative \(L^2\) error increases to \(2.00\times10^{-1}\). This divergence illustrates a key failure mode where the true transform structure may mimic the obstructive poles that the Pad\'e filters are intended to remove.

\begin{figure}[H]
 \centering
 \rmpbenchsheetgraphics{\benchfigHyperbolic}
 \caption{Hyperbolic molecular well benchmark sheet.}
 \label{fig:hyperbolic-dw}
\end{figure}

\endgroup
\FloatBarrier

The performance metrics summarized in Table~\ref{tab:bench_summary}, alongside the pointwise absolute error overlays (Figs.~\ref{fig:vr-overlay-coulomb}--\ref{fig:vr-overlay-hyperbolic}), emphasize that the efficacy of any inverse reconstruction is inherently potential-dependent. While the R{\"o}hrl-style LSQ baseline provides a robust variational comparator using 60 DD + 60 DN eigenvalue constraints (120 total), the Laplace--GBM framework consistently recovers the qualitative features of the radial ground state using significantly less data. Specifically, the recurrence-consumed subsets required---7 inputs for Coulomb and 22 for non-Coulomb cases---represent substantial data reductions that facilitate reconstruction when spectral information is scarce. All results are derived from archived benchmark runs with recorded settings and checksums, ensuring the exact reproducibility of the surfaces and closure families reported herein.

\begin{table}[H]
\centering
\caption{Relative $L^{2}$ values on five radial (1D) benchmarks for Laplace (GBM) and LSQ, with recurrence-consumed GBM input counts.}
\label{tab:bench_summary}
\small
% RMP-only benchmark summary table.
\begin{tabular}{@{}l c r r@{}}
  \toprule
  Potential & {\shortstack[c]{GBM\\inputs}} & {\shortstack[c]{LSQ\\(rel.\ $L^2$)}} & {\shortstack[c]{Laplace (GBM)\\(rel.\ $L^2$)}} \\
  \midrule
  Coulomb & 7 & 1.40e-1 & 1.91e-3 \\
  HO & 22 & 3.29e-2 & 7.10e-3 \\
  Hulth\'en & 22 & 1.68e-1 & 8.95e-3 \\
  Kratzer & 22 & 3.75e-1 & 4.94e-3 \\
  Hyperbolic molecular well & 22 & 5.40e-1 & 2.00e-1 \\
  \bottomrule
\end{tabular}
\par\smallskip
\footnotesize\emph{Note.} LSQ entries use 120 total eigenvalue constraints (60 Dirichlet--Dirichlet and 60 Dirichlet--Neumann). GBM inputs denote recurrence-consumed subsets under the active benchmark protocol (Coulomb 7; non-Coulomb 22). Values are rounded to match the active benchmark figure-legend precision.

\end{table}

\begingroup
\setlength{\intextsep}{2pt}
\setlength{\textfloatsep}{2pt}
\setlength{\floatsep}{2pt}
\setlength{\abovecaptionskip}{2pt}
\setlength{\belowcaptionskip}{0pt}

\begin{figure}[H]
 \centering
 \rmpvroverlaygraphics{\benchvrfigCoulomb}
 \caption{Coulomb \(V(r)\) overlay and pointwise absolute error for Exact, Laplace (GBM), and LSQ.}
 \label{fig:vr-overlay-coulomb}
\end{figure}

\begin{figure}[!htbp]
 \centering
 \rmpvroverlaygraphics{\benchvrfigHO}
 \caption{Harmonic oscillator \(V(r)\) overlay and pointwise absolute error for Exact, Laplace (GBM), and LSQ.}
 \label{fig:vr-overlay-ho}
\end{figure}

\begin{figure}[!htbp]
 \centering
 \rmpvroverlaygraphics{\benchvrfigHulthen}
 \caption{Hulth\'en \(V(r)\) overlay and pointwise absolute error for Exact, Laplace (GBM), and LSQ.}
 \label{fig:vr-overlay-hulthen}
\end{figure}

\begin{figure}[!htbp]
 \centering
 \rmpvroverlaygraphics{\benchvrfigKratzer}
 \caption{Kratzer \(V(r)\) overlay and pointwise absolute error for Exact, Laplace (GBM), and LSQ.}
 \label{fig:vr-overlay-kratzer}
\end{figure}

\begin{figure}[!htbp]
 \centering
 \rmpvroverlaygraphics{\benchvrfigHyperbolic}
 \caption{Hyperbolic molecular well \(V(r)\) overlay and pointwise absolute error for Exact, Laplace (GBM), and LSQ.}
 \label{fig:vr-overlay-hyperbolic}
\end{figure}

\endgroup
\FloatBarrier

% ---------- §4 Discussion & Practical Considerations ----------
\section{Discussion and practical considerations}
\label{sec:discussion}

The radial (1D) ground-state benchmarks demonstrate that the Laplace--GBM framework successfully bridges the gap between discrete spectral data and continuous potential reconstruction in the benchmark setting. By transforming limited energy gaps into ground-state density moments, the method provides a robust alternative to full spectral fitting, provided that the numerical stability of the analytic continuation is strictly managed. To transition from the theoretical benchmarks to a functional implementation, the following discussion details the specific protocols for data handling and stability control.

The foundational input-scaling stage involves establishing a consistent physical basis and addressing the inherent noise in empirical measurements. The reconstruction assumes a consistent Hartree atomic-units convention, utilizing a kinetic term of \(-\tfrac{1}{2}\nabla^2\). The primary inputs consist of low-lying energy gaps \(\{\Delta_\ell\}_{\ell\ge1}\) measured against the \(\ell=0\) ground state. These gaps form a truncated Laplace series for \(L(q)\), built from an even-moment ladder supplemented by an auxiliary interpolation to supply the missing odd-order moments of \(r^{2}\rho_{0,0}(r)\). When working with empirical data, pre-whitening through centering and rescaling is necessary to prevent single entries from dominating the Pad\'e selection. While full uncertainty quantification is outside the current scope, the dispersion of the recovered density \(\rho_{0,0}\) can be estimated by bootstrapping the input gaps through the ladder and Pad\'e stages.

Once the inputs are scaled, the framework must convert these discrete spectral gaps into a stable, continuous transform through analytic continuation. The GBM ladder inequality provides essential upper bounds for even moments, which stabilizes the resulting series for \(L(q)\). In practice, a shallow ladder of four to six moments is often sufficient, provided the input gaps are accurate. To ensure stability, researchers should monitor two diagnostics: successive GBM steps should not amplify numerical noise into alternating-sign estimates, and dropping the final ladder step should not drastically alter the selected Pad\'e family.

Because the truncated Taylor series is only reliable in a narrow neighborhood of \(q=0\), Pad\'e rational approximants \(P(N,D)\) are used to continue \(L(q)\) into the wider range required for stable inversion. Empirically, denominator-dominant Pad\'e families outperform universal diagonal rules for these benchmarks. To optimize results, we employ an obstructive-pole filter to discard candidates with singularities in the closed right half-plane, followed by a model-averaging step to mitigate noise and spurious pole-zero cancellations.

With a stable representation of \(L(q)\) established, the final stage of the pipeline requires numerical inversion and potential mapping. When the selected Pad\'e approximant is rational and well conditioned, the analytic inverse via residues is preferred, as it yields a fast and interpretable exponential sum. If the approximant is not simple, numerical Bromwich inversion---specifically schemes with proven convergence acceleration like de Hoog--Knight--Stokes Fourier-series acceleration or Abate--Whitt variants \cite{deHoogKnightStokes1982,AbateWhitt1995}---serves as a robust fallback.

Mapping the recovered reduced wavefunction \(\chi_{0,0}(r)\) back to the potential \(V(r)\) requires a monotone \(r\)-grid fine enough to resolve the shortest exponential scales. In Hartree units, the \(\ell=0\) recovery formula is \(V(r)=E_{0,0}+\tfrac{1}{2}\chi_{0,0}''(r)/\chi_{0,0}(r)\), with \(\chi_{0,0}(r)=r\sqrt{\rho_{0,0}(r)}\). Derivatives should be computed using a smoothing differentiator, such as a low-order Savitzky--Golay stencil, to suppress high-frequency noise without biasing the results near the origin. Finally, regularity at \(r=0\) must be enforced by ensuring \(\chi_{0,0}(0)=0\) and \(\chi_{0,0}(r)=\mathcal{O}(r)\) through either local polynomial fitting or extrapolation.

The following performance defaults and potential pitfalls provide a template for applying this method to future radial systems. Across the five benchmarks, low-order denominator-dominant fractions, typically with total order \(K\!\in\![6,10]\), together with strict pole filtering, provided the best bias--variance tradeoff. However, the method may encounter difficulties if the truncated series is too shallow to identify stable Pad\'e families, if gap estimates are severely flawed, or if the true transform contains branch cuts that mimic obstructive poles. In these instances, restoring stability often requires tightening the pole filter, shrinking the admissible index set, or acquiring additional moments.

To facilitate exact reproduction and fair comparison with least-squares baselines, reports should clearly document unit conventions, gap sets, ladder depth, the odd-moment completion family, and the specific Pad\'e filter rules, inversion route, numerical parameters, and differentiators employed.

% ---------- §5 Conclusion & Outlook ----------

\section{Conclusion and outlook}
\label{sec:conclusion}

This study presented a review-style synthesis and benchmark evaluation for reconstructing radial (1D) ground-state structures from restricted spectral data. By integrating a GBM-based even-moment ladder with Pad\'e-continued Laplace transforms, we established a robust benchmark pipeline for recovering both radial densities and their underlying potentials. The comparative analysis against the R{\"o}hrl-style least-squares baseline demonstrates that qualitatively accurate reconstructions are achievable from a compact subset of spectral inputs.

The reported rankings across five canonical potentials---Coulomb, harmonic oscillator, Hulth\'en, Kratzer, and the hyperbolic molecular well---remain empirical and specific to the stated benchmark configurations. Our results indicate that while the Laplace--GBM framework offers significant data efficiency, its performance is inherently potential dependent. Nevertheless, this manuscript provides a consistent mathematical frame for benchmark comparison, establishing standardized notation, unit conventions, and comparison surfaces for radial inverse problems.

Looking forward, several avenues exist to expand this framework. Future work will focus on incorporating explicit Fourier-side comparisons and broader cross-method evaluations. Additionally, further harmonization of notation and unit conventions across diverse potential classes will continue to refine the utility of moment-based inversion schemes in quantum spectral theory.
\appendix
% --- APPENDIX_NUMBERING_BEGIN ---
% Keep section refs as "Appendix A" while subordinate counters use local A.n.
\makeatletter
\@addtoreset{figure}{section}
\@addtoreset{table}{section}
\renewcommand{\theequation}{\Alph{section}.\arabic{equation}}
\renewcommand{\thefigure}{\Alph{section}.\arabic{figure}}
\renewcommand{\thetable}{\Alph{section}.\arabic{table}}
\renewcommand{\thetheorem}{\arabic{theorem}}
\renewcommand{\theassumption}{\Alph{section}.\arabic{assumption}}
\renewcommand{\thecorollary}{\Alph{section}.\arabic{corollary}}
\renewcommand{\thelemma}{\Alph{section}.\arabic{lemma}}
\renewcommand{\theproposition}{\Alph{section}.\arabic{proposition}}
\renewcommand{\thedefinition}{\Alph{section}.\arabic{definition}}
\renewcommand{\theexample}{\Alph{section}.\arabic{example}}
\renewcommand{\theremark}{\arabic{remark}}
\providecommand*{\theHfigure}{}
\providecommand*{\theHtable}{}
\providecommand*{\theHtheorem}{}
\providecommand*{\theHassumption}{}
\providecommand*{\theHcorollary}{}
\providecommand*{\theHlemma}{}
\providecommand*{\theHproposition}{}
\providecommand*{\theHdefinition}{}
\providecommand*{\theHexample}{}
\renewcommand*{\theHequation}{appendix.\Alph{section}.\arabic{equation}}
\renewcommand*{\theHfigure}{appendix.\Alph{section}.\arabic{figure}}
\renewcommand*{\theHtable}{appendix.\Alph{section}.\arabic{table}}
\renewcommand*{\theHtheorem}{appendix.\Alph{section}.\arabic{theorem}}
\renewcommand*{\theHassumption}{appendix.\Alph{section}.\arabic{assumption}}
\renewcommand*{\theHcorollary}{appendix.\Alph{section}.\arabic{corollary}}
\renewcommand*{\theHlemma}{appendix.\Alph{section}.\arabic{lemma}}
\renewcommand*{\theHproposition}{appendix.\Alph{section}.\arabic{proposition}}
\renewcommand*{\theHdefinition}{appendix.\Alph{section}.\arabic{definition}}
\renewcommand*{\theHexample}{appendix.\Alph{section}.\arabic{example}}
\renewcommand*{\theHremark}{appendix.\Alph{section}.\arabic{remark}}
\makeatother
% --- APPENDIX_NUMBERING_END ---

\refstepcounter{section}
\label{app:A}
\addcontentsline{toc}{section}{Appendix \Alph{section}}
\subsection*{Conditions on Some Potentials}
\phantomsection
\label{appA:conds}
\label{appA:conds-alt1}
Appendix A follows the foundational work of Bertlmann and Martin
\cite{Bertlmann}, which established key inequalities for heavy quark-antiquark
systems. While that work utilized these bounds to estimate quark masses, we
adapt and extend this framework here to provide the necessary constraints for
our inverse potential reconstruction. Specifically, we clarify the conditions
under which certain radial moment orderings hold---addressing areas noted as
unproven in the original treatment---and translate these theorems into the
radial--orbital notation \((n_r,\ell)\) used throughout this manuscript. While
certain results remain valid for completely arbitrary potentials, most of the
following analysis restricts the class of potentials \(V(r)\) to those
satisfying Conditions~(i)--(iii). These conditions are saturated by the
harmonic oscillator and Coulomb potentials.

\begin{equation}
\left(\frac{d}{d r}\right)^3 r^2 V \geq 0, \quad 
\lim _{r \rightarrow 0} r \left[V + \frac{1}{2} r \frac{d V}{d r} \right] \leq 0.
\tag{i}\label{appA-cond-A}
\end{equation}

\begin{equation}
\frac{d}{d r} \frac{1}{r} \frac{d}{d r}\left[V(r)+\frac{1}{2} r \frac{d V}{d r}\right]<0.
\tag{ii}\label{appA-cond-B}
\end{equation}

\begin{equation}
V^{\prime \prime}<0, \quad \text { i.e., } V \text { is concave.}
\tag{iii}\label{appA-cond-C}
\end{equation}

\begin{theorem}
\[
\left\langle r^2\right\rangle_{0,0} < \left\langle r^2\right\rangle_{0,1} < \left\langle r^2\right\rangle_{1,0}.
\]

The left-hand inequality holds for any potential, while the right-hand inequality holds if Conditions~(i) and~(ii) are satisfied.
\begin{remark}
The ordering \(\left\langle r^2\right\rangle_{0,1} < \left\langle r^2\right\rangle_{1,0}\) is not universal. A known counterexample is the infinite square well, where the \((0,1)\)-channel repulsion from the origin dominates, preventing the hierarchy from holding even if the energy ordering \(E_{0,1} < E_{1,0}\) remains intact.
\end{remark}
\begin{remark}
While Bertlmann and Martin expressed the belief that \(\left\langle r^2\right\rangle_{0,0} < \left\langle r^2\right\rangle_{1,0}\) should hold for any monotonically increasing potential (including the square well), a general proof was not obtained in their original work. This appendix identifies the specific conditions for which this ordering is verified, which include all attractive power potentials up to \(r^3\).
\end{remark}
\end{theorem}

\begin{proof}
\proofpart{Left-hand side}

With the reduced radial function \(u_{n_r,\ell}(r)=rR_{n_r,\ell}(r)\), the
mean-square radii reduce to one-dimensional radial integrals:
\begin{equation}
\langle r^2 \rangle_{0,0}=\int_{0}^{\infty} u_{0,0}^{2}r^2 dr,
\tag{A.1}\label{appA-eq-A-1}
\end{equation}

\begin{equation}
\langle r^2 \rangle_{0,1}=\int_{0}^{\infty} u_{0,1}^{2}r^2 dr.
\tag{A.2}\label{appA-eq-A-2}
\end{equation}

The corresponding Hartree-a.u.\ radial equations for the two lowest channels
are
\begin{equation}
\xRightarrow{\ell=0} -\frac{1}{2}u^{''}_{0,0} + Vu_{0,0} = E_{0,0}u_{0,0}.
\tag{A.3}\label{appA-eq-A-3}
\end{equation}

\begin{equation}
\xRightarrow{\ell=1} -\frac{1}{2}u^{''}_{0,1} + \frac{1}{r^2}u_{0,1} + Vu_{0,1} = E_{0,1}u_{0,1}.
\tag{A.4}\label{appA-eq-A-4}
\end{equation}

Let \(f(r)=u_{0,1}(r)/u_{0,0}(r)\). Since \(u_{0,0}>0\) for the ground state,
it is enough to control the Wronskian numerator
\[
W(r)=u_{0,0}u'_{0,1}-u_{0,1}u'_{0,0}.
\]
\begin{equation}
\frac{d}{dr}\left[\frac{u_{0,1}}{u_{0,0}}\right]
=\frac{W(r)}{u_{0,0}^{2}}>0 \quad (r>0),
\tag{A.5}\label{appA-eq-A-5}
\end{equation}

\begin{equation}
W'(r)=u_{0,0}u_{0,1}\left(\frac{2}{r^2}+2E_{0,0}-2E_{0,1}\right).
\tag{A.6}\label{appA-eq-A-6}
\end{equation}
Here \(W(0)=0\), and the centrifugal term makes \(W'(r)>0\) near the origin.
The Sturm comparison/non-crossing argument for the lowest radial states in the
\(\ell=0\) and \(\ell=1\) channels prevents \(W\) from crossing back through
zero; this is not an assertion that the \((0,1)\) state has an interior radial
node. Hence \(f\) is strictly increasing on \((0,\infty)\).

Because the two states are normalized and \(f(0)=0\), the strict monotonicity
of \(f\) implies that \(u_{0,1}^{2}-u_{0,0}^{2}\) has one zero, at
\(r=r_0\). The two factors in the next integral therefore change sign at the
same point:

\begin{equation}
\left<r^2\right>_{0,1} - \left<r^2\right>_{0,0} = \int^{\infty}_{0}(u^{2}_{0,1} - u^{2}_{0,0})(r^2 - r_0^2)dr > 0.
\tag{A.7}\label{appA-eq-A-7}
\end{equation}
This proves \(\left<r^2\right>_{0,0}<\left<r^2\right>_{0,1}\).

\proofpart{Right-hand side}
It remains to prove that
\begin{equation}
\left<r^2\right>_{1,0} - \left<r^2\right>_{0,1} = \int^{\infty}_{0}r^2(u^{2}_{1,0} - u^{2}_{0,1})dr.
\tag{A.8}\label{appA-eq-A-8}
\end{equation}

Let
\[
F(r)=V(r)+\frac{1}{2}rV'(r).
\]
By normalization, \(\int_{0}^{\infty}(u_{1,0}^{2}-u_{0,1}^{2})\,dr=0\). The
central-potential virial theorem gives \(2\langle T\rangle=\langle rV'\rangle\),
hence

\begin{equation*}
E_{1,0}=\int_{0}^{\infty}F(r)u_{1,0}^{2}\,dr,
  \quad
E_{0,1}=\int_{0}^{\infty}F(r)u_{0,1}^{2}\,dr.
\end{equation*}

Equivalently, if one applies the virial theorem to the radial effective form
\(V+\ell(\ell+1)/(2r^2)\), the virial derivative of the centrifugal term cancels
its static contribution, leaving the same central-potential expression above.
The energy-ordering hypothesis gives \(E_{1,0}>E_{0,1}\), so subtracting the
two virial identities shows

\begin{equation}
\int_{0}^{\infty}(u_{1,0}^2 - u_{0,1}^2)\Bigl[V + \frac{1}{2}r\frac{dV}{dr}\Bigr]dr>0.
\tag{A.9}\label{appA-eq-A-9}
\end{equation}

Bertlmann--Martin's comparison argument gives the corresponding two-crossing
sign pattern for \(u_{1,0}^{2}-u_{0,1}^{2}\). Under Condition~(ii), \(F(r)\)
is concave as a function of \(r^{2}\), so constants \(A\) and \(\mu>0\) may be
chosen such that \(F(r)-A-\mu r^{2}\) vanishes at the same two crossings. The
sign pattern then gives
\[
\int_{0}^{\infty}(u_{1,0}^{2}-u_{0,1}^{2})(F-A-\mu r^{2})\,dr<0.
\]
Using normalization and \eqref{appA-eq-A-9}, this inequality implies
\(\int_{0}^{\infty}r^{2}(u_{1,0}^{2}-u_{0,1}^{2})\,dr>0\), proving
\(\left<r^2\right>_{0,1}<\left<r^2\right>_{1,0}\). Together with the left-hand
inequality, the claimed hierarchy follows.

\end{proof}
\subsection*{Upper Bound for 2nd Moment of Position of Ground State}
\phantomsection
\label{sec:A.2}
\label{appA:gs-bound}
While Theorem~1 establishes the relative ordering of states, Theorem~2
turns the same framework into an upper bound for the ground-state second
moment. The bound uses the energy gap between the \(\ell=0\) and \(\ell=1\)
states to constrain the physical size of the ground-state density, which is
useful for keeping the inversion framework consistent when spectral input is
limited.

\begin{theorem}
The second moment of position for the ground state \((0,0)\) is bounded by the
energy difference between the \(\ell=0\) and \(\ell=1\) states:
\[
\left\langle r^2\right\rangle_{0,0}
\leq \frac{3}{2\left(E_{0,1}-E_{0,0}\right)}.
\]
\end{theorem}

\begin{proof}
The proof uses completeness together with the standard double-commutator
identity. Insert \(\mathds{1}=\sum_{k}|k\rangle\langle k|\) into the
expectation value of \([\hat{H},\hat{x}]\hat{x}\) to obtain
\begin{equation}
\begin{aligned}
\langle 0,0|[\hat{H},\hat{x}]\,\hat{x}|0,0\rangle
&=\sum_{k}\langle 0,0|[\hat{H},\hat{x}]|k\rangle\langle k|\hat{x}|0,0\rangle\\
&=\sum_{k}(E_{0,0}-E_{k})|\langle k|\hat{x}|0,0\rangle|^{2}.
\end{aligned}
\tag{A.10}\label{appA-eq-A-10}
\end{equation}
With $\hat{H}=\hat{p}^{2}/2+V$ and $\hat{p}_{x}=-i\partial_{x}$,
\begin{equation*}
[\hat{H},\hat{x}]=\frac{1}{2}[\hat{p}^{2},\hat{x}]
=\frac{1}{2}\left(\hat{p}_{x}[\hat{p}_{x},\hat{x}]+[\hat{p}_{x},\hat{x}]\,\hat{p}_{x}\right)
=-i\hat{p}_{x},
\end{equation*}
so
\begin{equation*}
[\hat{x},[\hat{H},\hat{x}]]=1.
\end{equation*}
Using $\langle 0,0|[\hat{x},[\hat{H},\hat{x}]]|0,0\rangle
 = 2\sum_{k}(E_{k}-E_{0,0})|\langle k|\hat{x}|0,0\rangle|^{2}$ yields
\begin{equation}
\sum_{k}(E_{k}-E_{0,0})|\langle k|\hat{x}|0,0\rangle|^{2}
=\frac{1}{2}.
\tag{A.11}\label{appA-eq-A-11}
\end{equation}
This is the one-coordinate sum rule. Summing over \(x,y,z\) gives
\[
\sum_{k}(E_{k}-E_{0,0})|\langle k|\mathbf{r}|0,0\rangle|^{2}=\frac{3}{2}.
\]
To establish the bound, keep only the lowest transition to an \(\ell=1\) state
with energy \(E_{0,1}\), thereby underestimating the positive sum. Using
\(\sum_{k}|\langle k|\mathbf{r}|0,0\rangle|^{2}=\langle r^{2}\rangle_{0,0}\)
then gives
\begin{equation}
\langle r^{2}\rangle_{0,0} \leq \frac{3}{2\left(E_{0,1}-E_{0,0}\right)}.
\tag{A.12}\label{appA-eq-A-12}
\end{equation}
This is the stated ground-state bound.

\end{proof}

% ---------- Appendix A third subheading ----------
\subsection*{Bounds for 2nd Moment of Position of the \((0,1)\) State}
\phantomsection
\label{appA:l1-bound}

Building upon the ground-state spatial constraint established in Theorem~2, we
use completeness and transition sum rules to derive analogous bounds for the
first angularly excited \((0,1)\) state.

\begin{theorem}
For any potential where Conditions~(i) and~(ii) are sufficient, and under the
energy ordering \(E_{0,1}<E_{1,0}<E_{0,2}\), the second moment
\(\left\langle r^2\right\rangle_{0,1}\) is bounded by
\[
\frac{2}{E_{0,1}-E_{0,0}} < \left\langle r^2\right\rangle_{0,1}
\leq \frac{3+\left(E_{1,0}-E_{0,0}\right)/\left(E_{0,1}-E_{0,0}\right)}{2\left(E_{1,0}-E_{0,1}\right)} .
\]
\end{theorem}

\begin{proof}\leavevmode
\proofpart{Right-hand side}
For the upper bound, start with the \((0,1)\) state in the radial (1D)
reduction (oriented along the x-axis). The TRK sum rule for the \(x\)
component, with the same one-component normalization as
Eq.~\eqref{appA-eq-A-11}, gives
\begin{equation}
\sum_{n}(E_n - E_{0,1})|\langle n|\hat{x}|0,1_x\rangle|^2 = \frac{1}{2},
\tag{A.13}\label{appA-eq-A-13}
\end{equation}
Separating out the ground-state contribution yields
\begin{equation}
\sum_{E_n > E_{0,1}}(E_n - E_{0,1})|\langle n|\hat{x}|0,1_x\rangle|^2
= \frac{1}{2} + (E_{0,1} - E_{0,0})|\langle 0,0|\hat{x}|0,1_x\rangle|^2,
\tag{A.14}\label{appA-eq-A-14}
\end{equation}
Let \(\Delta_{\min}^{(x)}:=\inf(E_{1,0} - E_{0,1}, E_{0,2} - E_{0,1})\). Then
\begin{equation}
\Delta_{\min}^{(x)}\sum_{E_n > E_{0,1}}|\langle n|\hat{x}|0,1_x\rangle|^2
\le \frac{1}{2} + (E_{0,1} - E_{0,0})|\langle 0,0|\hat{x}|0,1_x\rangle|^2.
\tag{A.15}\label{appA-eq-A-15}
\end{equation}
For \(y\) and \(z\), the ground state is forbidden by symmetry in this
orientation, so only \(\ell=2\) states contribute:
\begin{equation}
\sum_{E_n \ge E_{0,2}}(E_n - E_{0,1})|\langle n|\hat{y}|0,1_x\rangle|^2 = \frac{1}{2},
\qquad
\sum_{E_n \ge E_{0,2}}(E_n - E_{0,1})|\langle n|\hat{z}|0,1_x\rangle|^2 = \frac{1}{2}.
\tag{A.16}\label{appA-eq-A-16}
\end{equation}
\begin{equation}
\sum_{E_n \ge E_{0,2}}|\langle n|\hat{y}|0,1_x\rangle|^2
\le \frac{1}{2\left(E_{0,2} - E_{0,1}\right)},
\qquad
\sum_{E_n \ge E_{0,2}}|\langle n|\hat{z}|0,1_x\rangle|^2
\le \frac{1}{2\left(E_{0,2} - E_{0,1}\right)}.
\tag{A.17}\label{appA-eq-A-17}
\end{equation}
\begin{sloppypar}
Using completeness,
\[
\langle r^2\rangle_{0,1}=|\langle 0,0|\hat{x}|0,1_x\rangle|^2
\;+\sum_{E_n>E_{0,1}}|\langle n|\hat{x}|0,1_x\rangle|^2
\;+\sum_{E_n\ge E_{0,2}}\left(|\langle n|\hat{y}|0,1_x\rangle|^2+|\langle n|\hat{z}|0,1_x\rangle|^2\right).
\]
\end{sloppypar}
Applying \eqref{appA-eq-A-15}, \eqref{appA-eq-A-17}, and the ground-state dipole
bound established in the proof of Theorem~2,
\[
(E_{0,1}-E_{0,0})|\langle 0,0|\hat{x}|0,1_x\rangle|^2 \le \frac{1}{2},
\]
gives
\begin{equation}
\langle r^2\rangle_{0,1}
\le |\langle 0,0|\hat{x}|0,1_x\rangle|^2
+ \frac{\frac{1}{2} + (E_{0,1}-E_{0,0})|\langle 0,0|\hat{x}|0,1_x\rangle|^2}{\Delta_{\min}^{(x)}}
+ \frac{1}{E_{0,2}-E_{0,1}}.
\tag{A.18}\label{appA-eq-A-18}
\end{equation}
Under the ordering \(E_{0,1}<E_{1,0}<E_{0,2}\), we have
\(\Delta_{\min}^{(x)}=E_{1,0}-E_{0,1}\) and
\(E_{0,2}-E_{0,1}\ge E_{1,0}-E_{0,1}\). Therefore
\begin{equation}
\begin{split}
\langle r^2\rangle_{0,1}
&\le \frac{1}{2\left(E_{0,1}-E_{0,0}\right)}
\;+\frac{2}{E_{1,0}-E_{0,1}} \\
&= \frac{1}{2\left(E_{1,0}-E_{0,1}\right)}
\left[3 + \frac{E_{1,0}-E_{0,0}}{E_{0,1}-E_{0,0}}\right].
\end{split}
\tag{A.19}\label{appA-eq-A-19}
\end{equation}
which matches the RHS bound in the theorem.

\medskip
\proofpart{Left-hand side}
For the lower bound, use the monotonicity from Theorem~1: from
\eqref{appA-eq-A-5}, the ratio \(u_{0,1}(r)/u_{0,0}(r)\) is strictly
increasing for \(r>0\). Therefore, for the increasing weight \(r^2\),
\begin{equation}
\langle r^{2}\rangle_{0,1}
= \frac{\int_{0}^{\infty} r^{2}u_{0,1}^{2}\,dr}
{\int_{0}^{\infty}u_{0,1}^{2}\,dr}
> \frac{\int_{0}^{\infty} r^{2}u_{0,0}u_{0,1}\,dr}
{\int_{0}^{\infty}u_{0,0}u_{0,1}\,dr}.
\tag{A.20}\label{appA-eq-A-20}
\end{equation}
To evaluate the ratio, use the radial equations in Hartree a.u.\ for
\(\ell=0\) and \(\ell=1\), multiply the \(\ell=1\) equation by
\(r^{2}u_{0,0}\), the \(\ell=0\) equation by \(r^{2}u_{0,1}\), subtract, and
integrate by parts. This yields
\begin{equation}
\begin{split}
(E_{0,1}-E_{0,0})\int_{0}^{\infty} r^{2}u_{0,0}u_{0,1}\,dr
= 2\int_{0}^{\infty} u_{0,0}u_{0,1}\,dr \\
{}+ 2\int_{0}^{\infty} r\left(u_{0,1}'u_{0,0}-u_{0,0}'u_{0,1}\right)\,dr.
\end{split}
\tag{A.21}\label{appA-eq-A-21}
\end{equation}
The last integrand is strictly positive because \(u_{0,1}/u_{0,0}\) is
increasing. Hence
\begin{equation}
\begin{split}
(E_{0,1}-E_{0,0})\int_{0}^{\infty} r^{2}u_{0,0}u_{0,1}\,dr
> 2\int_{0}^{\infty} u_{0,0}u_{0,1}\,dr.
\end{split}
\tag{A.22}\label{appA-eq-A-22}
\end{equation}
Combining \eqref{appA-eq-A-20} and \eqref{appA-eq-A-22} gives the lower bound
\begin{equation}
\langle r^{2}\rangle_{0,1} \;>\; \frac{2}{E_{0,1}-E_{0,0}}.
\tag{A.23}\label{appA-eq-A-23}
\end{equation}
This proves the stated strict lower bound.
\end{proof}
% ---------- end Appendix A third subheading ----------

\refstepcounter{section}
\label{app:B}
\addcontentsline{toc}{section}{Appendix \Alph{section}}
\section*{Appendix \Alph{section}}\par
\subsection{Moment Completion and Analytic Continuation}
\label{sec:B}

\noindent While the generalized Bertlmann--Martin inequalities (GBMs) provide
a robust framework for reconstructing even moments, the Laplace-transform
pipeline requires a complete sequence of both even and odd moments to recover
the potential accurately in the reported benchmark setting. This appendix
explains why specific interpolation and continuation techniques---namely
moment-sequence completion and Pad\'e approximants---are necessary to bridge
the gap between discrete spectral data and continuous radial reconstruction.

The Bertlmann--Martin (BM) limitation relevant to odd-moment interpolation is
not that odd moments,
defined as
\[
\mu_n:=\langle r^n\rangle_{0,0}=\int_0^\infty r^{n+2}\,\rho_{0,0}(r)\,dr
=\int_0^\infty r^n\,\chi_{0,0}^2(r)\,dr
\]
vanish; they are physically positive
(e.g., \(\langle r\rangle>0\)). Rather, the limitation is structural: the BM
framework builds bounds from the positivity of energy--weighted spectral sums
involving multipole operators \(Q_{\ell m}=r^\ell Y_{\ell m}(\hat{\mathbf r})\)
and nested commutators with \(H\). After angular summation, the resulting scalar
constraints couple even multipole ranks and, consequently, even powers of \(r\).
These BMI-type recursions link \(\mu_{2k}\) to \(\mu_{2k-2}\) (and spectral
data) but contain no terms that connect to the odd subsequence
\(\{\mu_1,\mu_3,\dots\}\). Therefore, the odd moments must be supplied by an
auxiliary ansatz or interpolation.
In the canonical benchmark set reported here, that auxiliary family is
Pad\'e-complete for Coulomb, Hulth\'en, and Kratzer, and maximum-entropy for
HO and the hyperbolic molecular well.

Analytic continuation via Pad\'e approximants is needed because, when the
Laplace transform \(L(q)\) is constructed from finitely many moments,
the resulting power series
\[
L(q)=\sum_{n=0}^N a_n q^n + \mathcal O(q^{N+1}),\quad a_n=(-1)^n\mu_n/n!,
\]
provides only local Taylor data about \(q=0\). For a successful inversion
pipeline, the representation must remain stable over a wider \(q\)-range than
this moment-controlled neighborhood. The practical limitation is not that every
benchmark must have a finite radius of convergence: depending on the underlying
density, the exact transform may have nearby complex singularities or may be
entire. We replace the local series with a Pad\'e rational approximant
\(P(N,D)(q)=P_N(q)/Q_D(q)\) \cite{BakerGravesMorris1996PadeApproximants}, which
matches the known coefficients up to order \(N+D\) while providing a rational
continuation. This approach allows the transform, when singular structure is
present, to mimic singular features through poles and zeros, pushing past the
moment-limited window to recover \(r\)-space features without the severe bias
or instability seen in raw truncated series.

\refstepcounter{section}
\label{app:C}
\addcontentsline{toc}{section}{Appendix \Alph{section}}
\label{appendix:lsq_method}

\section*{Appendix \Alph{section}}\par
\subsection{The Least-Squares Variational Baseline}
\label{sec:C.1}
While the Laplace-based method utilizes a moment-driven pipeline, we assess its
performance against a more traditional spectral-fitting technique. This
appendix details the R{\"o}hrl-style Least Squares (LSQ) approach
\cite{Roehrl2006RecoveringBC, Roehrl2011LeastSquares},
a variational baseline that reconstructs the potential through direct
minimization of spectral mismatches. By comparing our results to this
established framework, we can better evaluate the efficiency and accuracy of
the GBM-driven reconstruction under shared benchmark conditions.

\subsection{The LSQ Functional}
\label{sec:C.2}
The R{\"o}hrl-style LSQ approach provides a variational baseline for
reconstructing the potential \(q(x)\) in the classical non-radial one-dimensional
Sturm--Liouville equation
\begin{equation}
  -u''(x) + q(x)u(x) = \lambda u(x), \quad x \in [0,1],
  \label{eq:sl_eq}
\end{equation}
from two spectra corresponding to different separated boundary conditions. Let
\((\alpha,\beta)\) and \((\alpha,\gamma)\) denote two sets of boundary angles,
with \(\sin(\beta-\gamma)\neq0\), and let \(\lambda_{q,i,n}\) be the \(n\)-th
eigenvalue for boundary condition set \(i\in\{1,2\}\). Borg's and Levinson's
theorems guarantee uniqueness of \(q(x)\) given the full two spectra. Given
partial spectral data \(\{\lambda_{Q,i,n}\}_{(i,n)\in I}\) from an unknown
potential \(Q(x)\), the method minimizes the functional
\begin{equation}
  G[q] = \sum_{(i,n)\in I} \omega_{i,n} \left( \lambda_{q,i,n} - \lambda_{Q,i,n} \right)^2,
  \label{eq:lsq_functional}
\end{equation}
where \(\omega_{i,n}>0\) are user-defined weights (often
\(\omega_{i,n}=1\)). The functional satisfies \(G[q]\geq0\) and \(G[q]=0\) if
and only if \(\lambda_{q,i,n}=\lambda_{Q,i,n}\) for all \((i,n)\in I\).

\subsection{Gradient Descent and Convexity}
\label{sec:C.3}
Let \(g_{q,i,n}(x)\) denote the normalized eigenfunction corresponding to
\(\lambda_{q,i,n}\). The Fréchet derivative of the eigenvalue in direction
\(h(x)\) is
\begin{equation}
  \frac{\partial}{\partial q} \lambda_{q,i,n}[h] = \int_{0}^{1} h(x)\, g_{q,i,n}^2(x) \, dx.
  \label{eq:eig_derivative}
\end{equation}
allowing the gradient of \(G\) to be written as
\begin{equation}
  \nabla G[q](x) = 2 \sum_{(i,n)\in I} \omega_{i,n} \left( \lambda_{q,i,n} - \lambda_{Q,i,n} \right) g_{q,i,n}^2(x).
  \label{eq:lsq_gradient}
\end{equation}
The squared eigenfunctions \(\{g_{q,i,n}^{2}\}\) are linearly independent in
\(H^1[0,1]\). In an idealized formulation, \(\nabla G[q]=0\) if and only if
\(G[q]=0\), supporting gradient-based minimization as a principled baseline.
Finite-grid and finite-data implementations still require empirical convergence
checks and sensitivity audits.

\subsection{Numerical Minimization Procedure}
\label{sec:C.4}
In practice, \(G[q]\) is minimized via the Polak--Ribière conjugate gradient
algorithm through a structured iterative process. The routine begins by
initializing a trial potential \(q^{(0)}(x)\), which is frequently set to
\(q^{(0)}\equiv0\). At each iteration \(k\), the algorithm computes the
eigenvalues and eigenfunctions for both boundary condition sets to evaluate the
functional \(G[q^{(k)}]\) and its corresponding gradient \(\nabla G[q^{(k)}]\).
A search direction is then determined using conjugate gradients, followed by an
Armijo line search to identify the optimal step length \(\alpha_k\). Finally,
the potential is updated according to \(q^{(k+1)}=q^{(k)}+\alpha_k d_k\), and
the cycle repeats until the objective functional falls below a predefined
tolerance or the maximum iteration count is reached.

\begingroup
% April 30 coauthor numbering note: render these Appendix C support floats in
% the whole-paper table/figure sequence while leaving Appendix equations local.
\renewcommand{\thetable}{\arabic{table}}
\renewcommand{\thefigure}{\arabic{figure}}
\renewcommand*{\theHtable}{appendix.C.global.\arabic{table}}
\renewcommand*{\theHfigure}{appendix.C.global.\arabic{figure}}
\setcounter{table}{1}
\setcounter{figure}{10}
\setlength{\intextsep}{2pt}
\setlength{\textfloatsep}{2pt}
\setlength{\floatsep}{2pt}
\setlength{\abovecaptionskip}{2pt}
\setlength{\belowcaptionskip}{0pt}

\begin{table}[H]
\centering
\caption{Benchmark comparison between the R{\"o}hrl-style LSQ baseline and Laplace (GBM) on the Coulomb benchmark.}
% Shared Coulomb comparison table (RMP appendix).
\begin{tabular}{lcc}
  \hline
  \textbf{Metric} & \textbf{R{\"o}hrl-style LSQ} & \textbf{Laplace (GBM)} \\
  \hline
  Iterations & 300 & -- \\
  Elapsed time (s) & 7.433442e+01 & 1.259797e+01 \\
  L2 (q\_rec, q\_true) & 6.417976e-04 & 9.586654e-04 \\
  RMSE eigenvalues (DD) & 2.543589e-05 & 2.019329e-04 \\
  RMSE eigenvalues (DN) & 3.594460e-05 & 2.016194e-04 \\
  Final F & 4.910380e-10 & -- \\
  Final eig. misfit (both sets) & 8.703796e-05 & 2.553403e-04 \\
  \hline
\end{tabular}

\end{table}

\begin{figure}[H]
  \centering
  \rmpincludegraphics[width=0.78\linewidth]{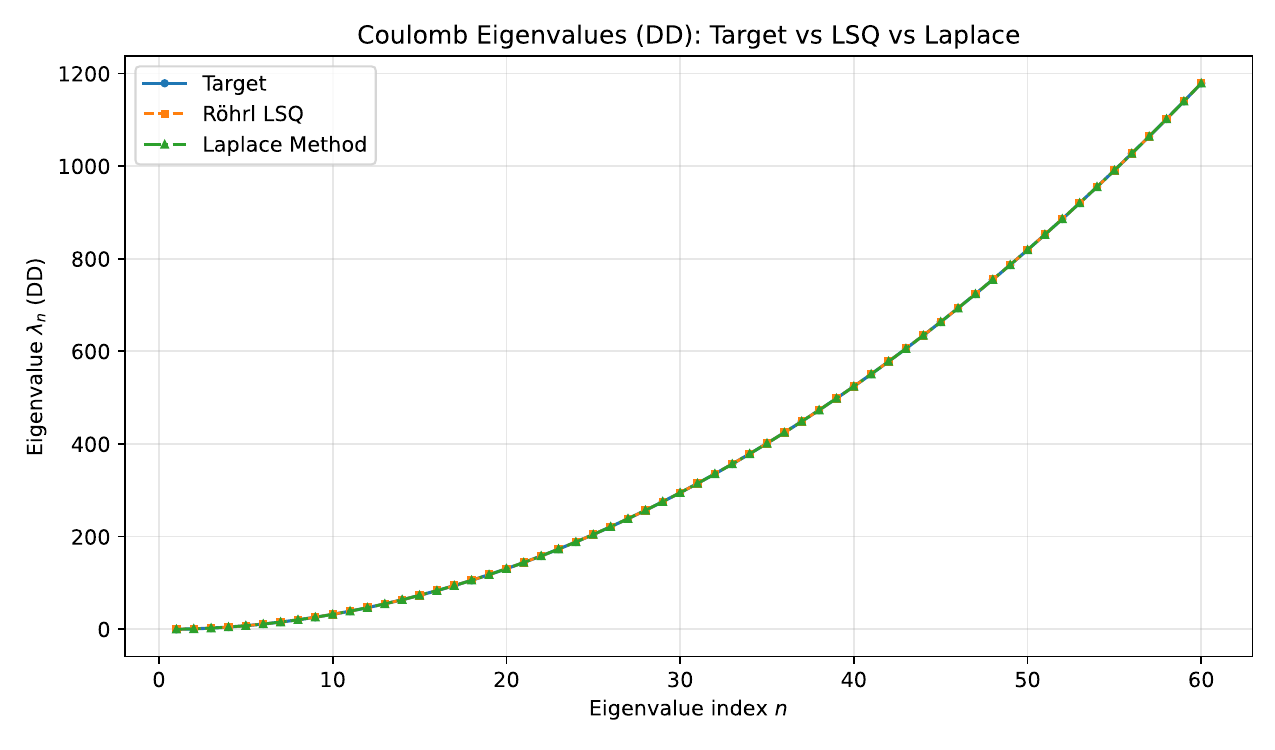}
  \caption{Coulomb DD eigenvalue fit for the R{\"o}hrl-style LSQ baseline.}
  \label{fig:C_DD_LS}
\end{figure}

\begin{figure}[H]
  \centering
  \rmpincludegraphics[width=0.78\linewidth]{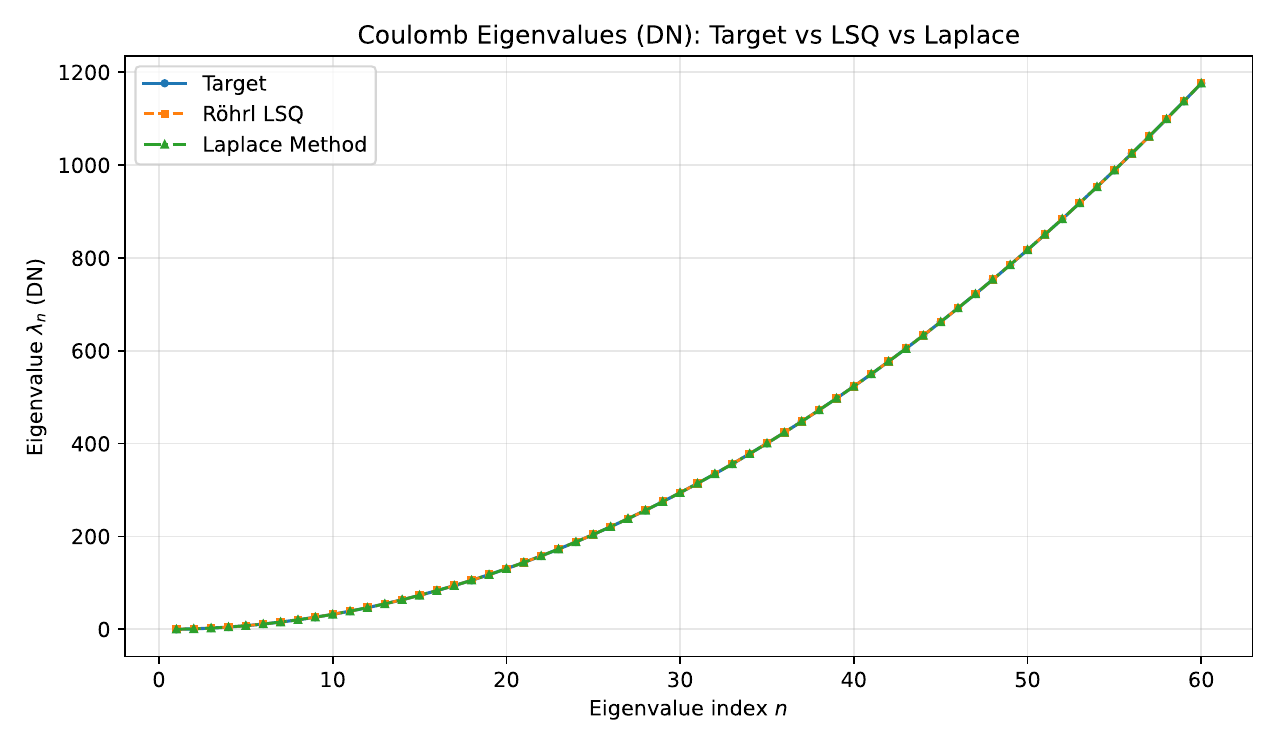}
  \caption{Coulomb DN eigenvalue fit for the R{\"o}hrl-style LSQ baseline.}
  \label{fig:C_DN_LS}
\end{figure}
\endgroup

\subsection{Implementation in the Benchmark Script}
\label{sec:C.5}
The benchmarking framework implements the LSQ functional using a discrete grid
representation of the potential \(q(x)\). It utilizes a finite-difference
assembly of the discrete Sturm--Liouville operator to handle both
Dirichlet--Dirichlet and Dirichlet--Neumann boundary conditions. The script
performs direct diagonalization to obtain the eigenvalues \(\lambda_{q,i,n}\)
and the eigenfunctions \(g_{q,i,n}\) required for gradient evaluation according
to Eq.~\eqref{eq:lsq_gradient}. Additionally, the implementation allows for an
optional regularization term \(\mu\int (q'(x))^2\,dx\) to ensure that the
reconstructed potential remains smooth.

\subsection{Role as a Benchmark}
\label{sec:C.6}
The LSQ method serves as a reproducible reference baseline against the
Laplace-transform-based inverse method developed in this work. Unlike the
Laplace approach, which reconstructs the radial potential \(V(r)\) from
ground-state density moments constrained by Bertlmann--Martin inequalities, the
LSQ method directly fits the spectral data in a two-spectra setting. It is a
valuable benchmark because it targets the exact potential through direct
spectral-misfit minimization and provides a complementary reconstruction family
with distinct assumptions and failure modes. Furthermore, it requires no
assumptions regarding moment sequences or analytic continuation. Comparative
conclusions are drawn from the reported metrics and figures for the stated
benchmark settings. For benchmark-count comparisons, the LSQ method is treated
as having 120 total constraints (60 DD + 60 DN), while the Laplace/GBM counts
refer only to the recurrence-consumed subset of bound-state inputs.

\begingroup
% April 30 coauthor numbering note: render these Appendix C support floats in
% the whole-paper table/figure sequence while leaving Appendix equations local.
\renewcommand{\thefigure}{\arabic{figure}}
\renewcommand*{\theHfigure}{appendix.C.global.\arabic{figure}}
\setcounter{figure}{12}
\setlength{\intextsep}{2pt}
\setlength{\textfloatsep}{2pt}
\setlength{\floatsep}{2pt}
\setlength{\abovecaptionskip}{2pt}
\setlength{\belowcaptionskip}{0pt}

\begin{figure}[H]
  \centering
  \rmpincludegraphics[width=0.9\linewidth]{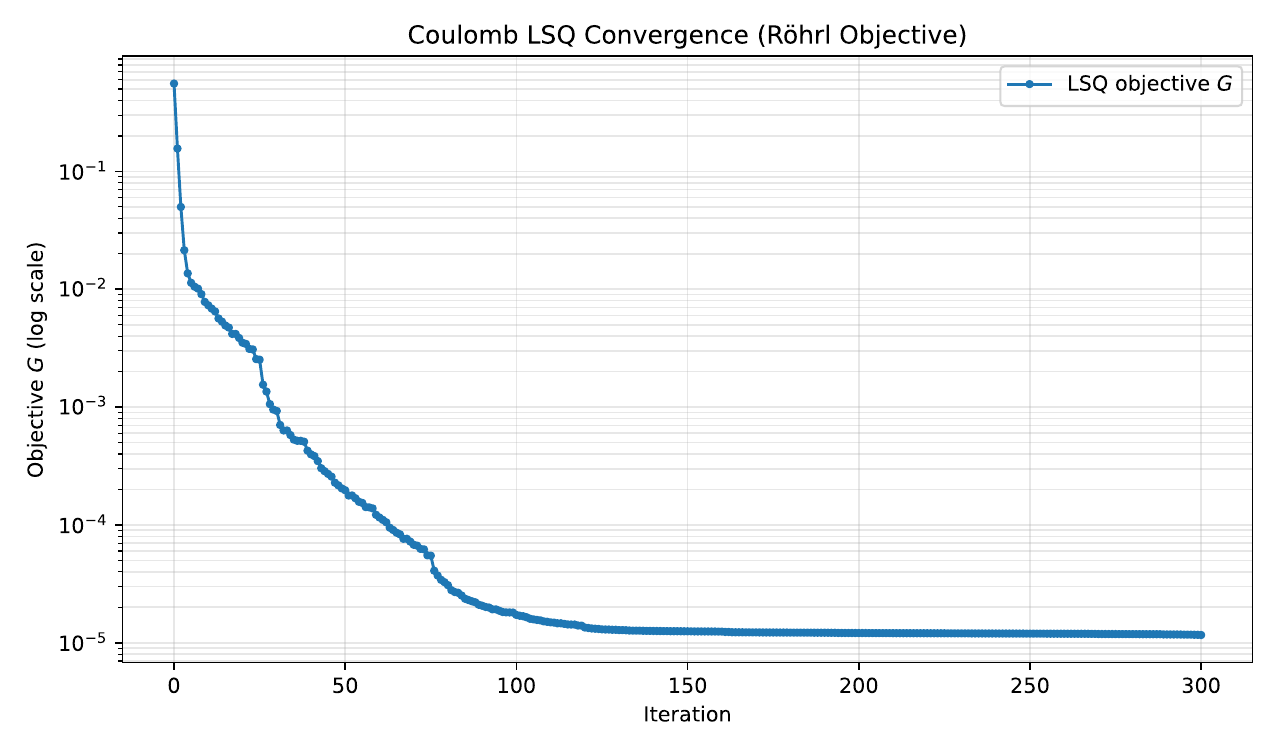}
  \caption{Coulomb LSQ convergence for the R{\"o}hrl-style two-spectra inverse Sturm--Liouville functional. The curve plots the objective \(F(c)=G(c)+\mu R(c)\) versus nonlinear conjugate-gradient (Polak--Ribière) iteration, where \(G\) is the weighted sum of squared DD/DN eigenvalue mismatches and \(R\) is an optional smoothing penalty on \(q'\).}
  \label{fig:C_LS_Conv}
\end{figure}
\endgroup

%% --- References ---
\bibliographystyle{ws-rmp}
\bibliography{citations_rmp}

\begin{thebibliography}{10}

\bibitem{DoerflerHochbruckKoehlerRiederSchnaubeltWieners2023}
W.~D{\"o}rfler, M.~Hochbruck, J.~K{\"o}hler, A.~Rieder, R.~Schnaubelt and
  C.~Wieners, {\em Wave Phenomena: Mathematical Analysis and Numerical
  Approximation}, Oberwolfach Seminars, Vol.~49 (Birkh{\"a}user, Cham, 2023).

\bibitem{Ambarzumian}
V.~Ambarzumian, {{\"U}ber eine Frage der Eigenwerttheorie}, {\em Zeitschrift
  f{\"u}r Physik} {\bf 53}(9-10)  (1929)  690--695.

\bibitem{Borg}
G.~Borg, {Eine Umkehrung der Sturm-Liouvilleschen Eigenwertaufgabe: Bestimmung
  der Differentialgleichung durch die Eigenwerte}, {\em Acta Mathematica} {\bf
  78}  (1946)  1--96.

\bibitem{Levinson}
N.~Levinson, The inverse {Sturm-Liouville} problem, {\em Matematisk Tidsskrift
  B} {\bf 1949}  (1949)  25--30.

\bibitem{Tikhonov}
A.~N. Tikhonov, {{\"U}ber die Eindeutigkeit der L{\"o}sung des Problems der
  Elektro-Sch{\"u}rfung}, {\em Doklady Akademii Nauk SSSR, Novaya Seriya} {\bf
  69}  (1949)  797--800.

\bibitem{Levitan1987InverseSturmLiouville}
B.~M. Levitan, {\em Inverse {Sturm--Liouville} Problems} (VSP, Zeist, 1987).
\newblock Translated from the Russian by O. Efimov.

\bibitem{FreilingYurko2001InverseSturmLiouville}
G.~Freiling and V.~Yurko, {\em Inverse {Sturm--Liouville} Problems and Their
  Applications} (Nova Science Publishers, Huntington, NY, 2001).

\bibitem{MamedovCetinkaya2013SpectralParameter}
K.~R. Mamedov and F.~A. Cetinkaya, Inverse problem for a class of
  {Sturm--Liouville} operator with spectral parameter in boundary condition,
  {\em Boundary Value Problems} {\bf 2013}  (2013) p. 183.

\bibitem{Bondarenko2015MatrixHalfLine}
N.~Bondarenko, An inverse spectral problem for the matrix {Sturm--Liouville}
  operator on the half-line, {\em Boundary Value Problems} {\bf 2015}  (2015)
  p.~15.

\bibitem{Chen2016InteriorTransmission}
L.-H. Chen, Inverse uniqueness in interior transmission problem and its
  eigenvalue tunneling in simple domain, {\em Advances in Mathematical Physics}
  {\bf 2016}  (2016) p. 2438253.

\bibitem{ButerinSat2017HalfInverseIntegro}
S.~A. Buterin and M.~Sat, On the half inverse spectral problem for an
  integro-differential operator, {\em Inverse Problems in Science and
  Engineering} {\bf 25}(10)  (2017)  1508--1518.

\bibitem{DelgadoKhmelnytskayaKravchenko2019}
B.~B. Delgado, K.~V. Khmelnytskaya and V.~V. Kravchenko, The transmutation
  operator method for efficient solution of the inverse {Sturm--Liouville}
  problem on a half-line, {\em Mathematical Methods in the Applied Sciences}
  {\bf 42}(18)  (2019)  7359--7366.

\bibitem{ZhangBondarenkoYang2021Discontinuous}
R.~Zhang, N.~P. Bondarenko and C.~F. Yang, Solvability of an inverse problem
  for discontinuous {Sturm--Liouville} operators, {\em Mathematical Methods in
  the Applied Sciences} {\bf 44}(1)  (2021)  124--139.

\bibitem{VladicicBoskovicVojvodic2022Delay}
V.~Vladi{\v c}i{\'c}, M.~Bo{\v s}kovi{\'c} and B.~Vojvodi{\'c}, Inverse
  problems for {Sturm--Liouville}-type differential equation with a constant
  delay under {Dirichlet}/polynomial boundary conditions, {\em Bulletin of the
  Iranian Mathematical Society} {\bf 48}(4)  (2022)  1829--1843.

\bibitem{WeiHuXiang2024Reconstruction}
Z.~Wei, Z.~Hu and Y.~Xiang, Reconstruction of the solution of inverse
  {Sturm--Liouville} problem, {\em Boundary Value Problems} {\bf 2024}  (2024)
  p.~55.

\bibitem{GawishMansour2025QSturmLiouville}
F.~A. Gawish and Z.~S. Mansour, On uniqueness theorems for the inverse
  {$q$}-{Sturm--Liouville} problems, {\em Quaestiones Mathematicae} {\bf 48}(7)
   (2025)  1065--1096.

\bibitem{FeizmohammadiKian2026PartialBoundary}
A.~Feizmohammadi and Y.~Kian, Recovery of {Sturm--Liouville} operators from
  partial boundary spectral data and applications  (2026), Annales Henri
  Poincar{\'e}, online first.

\bibitem{GelfandLevitan}
I.~M. Gel'fand and B.~M. Levitan, On the determination of a differential
  equation from its spectral function, {\em American Mathematical Society
  Translations, Series 2} {\bf 1}  (1955)  253--304.

\bibitem{Marchenko}
V.~A. Marchenko, {Some questions of the theory of one-dimensional linear
  differential operators of the second order. I, II}, {\em Trudy Moskovskogo
  Matematicheskogo Obshchestva} {\bf 1--2}  (1952--1953)  327--420; 3--83.

\bibitem{HronRazavy1983GLMApplications}
M.~Hron and M.~Razavy, Some applications of the {Gel'fand--Levitan} inverse
  method in atomic and molecular physics, {\em International Journal of Quantum
  Chemistry} {\bf 24}(1)  (1983)  97--111.

\bibitem{Habashy1991GeneralizedGLM}
T.~M. Habashy, A generalized {Gel'fand--Levitan--Marchenko} integral equation,
  {\em Inverse Problems} {\bf 7}(5)  (1991)  703--711.

\bibitem{KravchenkoTorba2021DirectMethod}
V.~V. Kravchenko and S.~M. Torba, A direct method for solving inverse
  {Sturm--Liouville} problems, {\em Inverse Problems} {\bf 37}(1)  (2021) p.
  015015.

\bibitem{KravchenkoVicenteBenitez2022TransmutationII}
V.~V. Kravchenko and V.~A. Vicente-Ben{\'i}tez, Transmutation operators method
  for {Sturm--Liouville} equations in impedance form {II}: Inverse problem,
  {\em Journal of Mathematical Sciences} {\bf 266}(4)  (2022)  554--575.

\bibitem{BaevMozgovykh2025MethodSolving}
A.~V. Baev and V.~V. Mozgovykh, A method for solving the inverse
  {Sturm--Liouville} problem, {\em Moscow University Computational Mathematics
  and Cybernetics} {\bf 49}(3)  (2025)  173--181.

\bibitem{Chadan}
K.~Chadan and P.~C. Sabatier, {\em Inverse Problems in Quantum Scattering
  Theory} (Springer, 2012).

\bibitem{Newton}
R.~G. Newton, {\em Scattering Theory of Waves and Particles}, 2nd edn.
  (Springer, 2013).

\bibitem{KravchenkoShishkinaTorba2020InverseScattering}
V.~V. Kravchenko, E.~L. Shishkina and S.~M. Torba, A transmutation operator
  method for solving the inverse quantum scattering problem, {\em Inverse
  Problems} {\bf 36}(12)  (2020) p. 125007.

\bibitem{GradusovYakovlev2023InverseSquareScattering}
V.~A. Gradusov and S.~L. Yakovlev, On the scattering problem for a potential
  decreasing as the inverse square of distance, {\em Theoretical and
  Mathematical Physics} {\bf 217}(2)  (2023)  1777--1787.

\bibitem{Gibson2026ExplicitInverseScattering}
P.~C. Gibson, Explicit inverse scattering for the one-dimensional
  {Schr{\"o}dinger} equation, {\em Mathematical Models and Methods in Applied
  Sciences} {\bf 36}(4)  (2026)  751--785.

\bibitem{FabianoKnobelLowe1995FiniteDifference}
R.~H. Fabiano, R.~Knobel and B.~D. Lowe, A finite-difference algorithm for an
  inverse {Sturm--Liouville} problem, {\em IMA Journal of Numerical Analysis}
  {\bf 15}(1)  (1995)  75--88.

\bibitem{MarlettaWeikard2005WeakStability}
M.~Marletta and R.~Weikard, Weak stability for an inverse {Sturm--Liouville}
  problem with finite spectral data and complex potential, {\em Inverse
  Problems} {\bf 21}(4)  (2005)  1275--1290.

\bibitem{FreilingMazurYurko2007Singular}
G.~Freiling, T.~Mazur and V.~Yurko, A numerical algorithm for solving inverse
  problems for singular {Sturm--Liouville} operators, {\em Advances in
  Dynamical Systems and Applications} {\bf 2}(1)  (2007)  95--105.

\bibitem{Andrew2011FiniteDifference}
A.~L. Andrew, Finite difference methods for half inverse {Sturm--Liouville}
  problems, {\em Applied Mathematics and Computation} {\bf 218}(2)  (2011)
  445--457.

\bibitem{GaoHuangCheng2015FiniteDifference}
Q.~Gao, Z.~Huang and X.~Cheng, A finite difference method for an inverse
  {Sturm--Liouville} problem in impedance form, {\em Numerical Algorithms} {\bf
  70}(3)  (2015)  669--690.

\bibitem{Kravchenko2020DirectInverseBook}
V.~V. Kravchenko, {\em Direct and Inverse {Sturm--Liouville} Problems: A Method
  of Solution} (Springer, 2020).

\bibitem{Bondarenko2022FrozenArgument}
N.~P. Bondarenko, Finite-difference approximation of the inverse
  {Sturm--Liouville} problem with frozen argument, {\em Applied Mathematics and
  Computation} {\bf 413}  (2022) p. 126653.

\bibitem{CetinkayaKhmelnytskayaKravchenko2024NSBF}
F.~A. {\c C}etinkaya, K.~V. Khmelnytskaya and V.~V. Kravchenko, Neumann series
  of {Bessel} functions for inverse coefficient problems, {\em Mathematical
  Methods in the Applied Sciences} {\bf 47}(16)  (2024)  12373--12387.

\bibitem{Kravchenko2025WeylHalfLine}
V.~V. Kravchenko, Reconstruction techniques for inverse {Sturm--Liouville}
  problems with complex coefficients, {\em Mathematical Methods in the Applied
  Sciences} {\bf 48}(17)  (2025)  15875--15889.

\bibitem{Bertlmann}
R.~A. Bertlmann and A.~Martin, Inequalities on heavy quark--antiquark systems,
  {\em Nuclear Physics B} {\bf 168}  (1980)  111--136.

\bibitem{Mezhoud2003GBM}
R.~Mezhoud, F.-Z. Ighezou, A.~Chouchaoui, A.~T. Kerris and R.~J. Lombard,
  {Generalized Bertlmann--Martin Inequalities and Power-Law Potentials}, {\em
  Annals of Physics} {\bf 308}(1)  (2003)  143--155.

\bibitem{Yekken}
R.~Yekken, F.-Z. Ighezou and R.~J. Lombard, The inverse problem in the case of
  bound states, {\em Annals of Physics} {\bf 323}(1)  (2008)  61--81.

\bibitem{Mezhoud}
R.~Mezhoud, I.~Ami and R.~J. Lombard, The use of {Laplace} transform in the
  inverse problem from bound states, {\em Romanian Reports in Physics} {\bf
  75}(3)  (2023) p. 114.

\bibitem{deHoogKnightStokes1982}
F.~R. {de Hoog}, J.~H. Knight and A.~N. Stokes, An improved method for
  numerical inversion of {Laplace} transforms, {\em SIAM Journal on Scientific
  and Statistical Computing} {\bf 3}(3)  (1982)  357--366.

\bibitem{AbateWhitt1995}
J.~Abate and W.~Whitt, Numerical inversion of {Laplace} transforms of
  probability distributions, {\em ORSA Journal on Computing} {\bf 7}(1)  (1995)
   36--43.

\bibitem{Stieltjes1894FractionsContinues}
T.~J. Stieltjes, Recherches sur les fractions continues, {\em Annales de la
  Facult\'e des sciences de Toulouse : Math\'ematiques} {\bf 8}(4)  (1894)
  J1--J122, S\'erie 1.

\bibitem{Roehrl2006RecoveringBC}
N.~R{\"o}hrl, Recovering boundary conditions in inverse {Sturm-Liouville}
  problems, in {\em Recent Advances in Differential Equations and Mathematical
  Physics, Contemporary Mathematics, Vol.~412\/}  (American Mathematical
  Society, 2006), pp. 263--270.

\bibitem{Roehrl2011LeastSquares}
N.~R{\"o}hrl, A least-squares functional for solving inverse {Sturm-Liouville}
  problems, {\em Inverse Problems} {\bf 21}(6)  (2005)  2009--2017.

\bibitem{BakerGravesMorris1996PadeApproximants}
G.~A. Baker and P.~Graves-Morris, {\em Pad{\'e} Approximants}, Encyclopedia of
  Mathematics and its Applications, Vol.~59, 2 edn. (Cambridge University
  Press, 1996).

\end{thebibliography}

\end{document}